\documentclass[11pt]{article}

\usepackage{amsfonts,amsmath,amsthm,amscd,amssymb,latexsym,amsbsy,pb-diagram,cite,caption,url}

\usepackage{tikz}
\usetikzlibrary{positioning}
\usepackage{float}
\usepackage{circledtext}

\theoremstyle{plain}
\newtheorem{theorem}{Theorem}[section]

\newtheorem{proposition}{Proposition}[section]
\newtheorem{corollary}{Corollary}[section]
\newtheorem{definition}{Definition}[section]
\newtheorem{conjecture}[theorem]{Conjecture}

\theoremstyle{definition}
\newtheorem{example}{Example}[section]
\newtheorem{remark}{Remark}[section]
\newtheorem{problem}{Problem}[section]
\newcommand{\keywords}{\textbf{Key words. }\medskip}
\newcommand{\subjclass}{\textbf{MSC 2020. }\medskip}
\renewcommand{\abstract}{\textbf{Abstract. }\medskip}

\numberwithin{equation}{section}

\begin{document}

\title{Ultrametric spaces generated by labeled star graphs}

\author{Oleksiy Dovgoshey, Olga Rovenska}



\date{}

\maketitle

\begin{abstract}
 For arbitrary star graph $S$ with a non-degenerate vertex labeling $l\colon V(S) \to \mathbb{R}^+$ we denote by $d_l$ the corresponding ultrametric  on the vertex set $V(S)$ of $S$.
We characterize the class $\bf US$ of all ultrametric spaces $(V(S), d_l)$ up to isometry. We also find the necessary and sufficient conditions under which
the group of all self-isometries of ultrametric space $(V(S), d_l)$ coincides with the group of all self-isomorphisms of the labeled star graph $S(l)$.
\end{abstract}



\section{Introduction}

\hspace{4mm}
The problem of describing of finite ultrametric spaces using graph theory was posed in \cite{Lemin2001} by I.\,M.\,Gelfand and solved in \cite{Gurvich-Vyalyi2012} by V.\,Gurvich and M.\,Vyalyi with the help of monotone trees. A geometric interpretation of Gurvich-Vyalyi  representing trees was found in \cite{Petrov-Dovgoshey2014}. This allows us to use the Gurvich-Vyalyi representation in various extremal problems connected with finite ultrametric spaces \cite{Dovgoshey-Petrov2020,Dovgoshey-Petrov-Teichert2015,Dovgoshey-Petrov-Teichert2017}.

The Gurvich-Vyalyi trees form a special class of finite trees endowed with vertex labelings. The combinatorial and topological properties of infinite trees endowed with positive edge labelings have been studied in \cite{BD2006CPC, BDS2005JGT, BS2010CPC, DIESTEL2006846, DIESTEL20111423, Deistel2017, DK2004EJC, DP2017JCT, DS2011AM, DS2011TIA, DS2012DM, DMV2013AC, DP2013SM}.

The ultrametric spaces generated by so-called non-generate vertex labeling on arbitrary trees were first introduced in \cite{Dov2020TaAoG} and studied in \cite{Dovgoshey-Küçükaslan2022,Dovgoshey-Kostikov-2023}. The simplest types of infinite trees are rays and star graphs. The totally bounded ultrametric spaces generated by labeled  almost rays have recently been characterized in \cite{Dovgoshey-Vito}.

Let us denote by $\mathbb{R}^+$ the set $[0, \infty).$

An {\it ultrametric} on a nonempty set $X$ is a symmetric function
$
d \colon X \times X \to \mathbb{R}^+
$
such that we have $d(x, y) = 0$ if and only if $x = y$  and, moreover, the {\it strong triangle inequality}
\begin{equation}
\label{d(x,y)}
    d(x, y) \leq \max\{d(x, z),\,d(z,y)\}
\end{equation}
      holds for all $x, y, z \in X$.

A sequence $(x_n)_{n \in \mathbb{N}}$ of points in an ultrametric space $(X, d)$ is said to {\it converge} to a point $p \in X$,
\begin{equation}
    \label{sa}
\lim_{n \to \infty} x_n = p,
\end{equation}
if  $
\lim\limits_{n \to \infty} d(x_n, p) = 0 $
holds. If \eqref{sa} holds, then we say that $p$ is the {\it limit} of $(x_n)_{n \in \mathbb{N}}$ in $(X, d)$.

\begin{definition}
\label{am}
Let $(X, d)$ and $(Y, \rho)$ be ultrametric spaces. A bijective mapping
$
F\colon X \to Y
$
is said to be an  {\it isometry} if the equality
\begin{equation*}
d(x, y)=\rho\left(F(x), F(y)\right)
\end{equation*}
holds for all $x, y \in X$.
\end{definition}

The ultrametric spaces are said to be {\it isometric} if there is an isometry of these spaces.

Let us recall now some concepts from the graph theory.

The \textit{simple graph} is a pair $(V, E)$ consisting of a nonempty set $V$ and a set $E$ whose elements are unordered pairs $\{u, v\}$ of different points $u, v \in V$. For a graph $G = (V, E)$, the sets $V = V(G)$ and $E = E(G)$ are called the \textit{set of vertices} and the \textit{set of edges}, respectively. A graph $G$ is \textit{finite} if $V(G)$ is a finite set. A graph $H$ is, by definition, a \textit{subgraph} of a graph $G$ if the inclusions $V(H) \subseteq V(G)$ and $E(H) \subseteq E(G)$ are valid. In this case, we simply write $H \subseteq G$.

In what follows we will use the standard definitions of paths and cycles, see, for example, \cite[Section 1.3]{Diestel2017}. A graph $G$ is {\it connected} if for every two distinct $u, v \in V(G)$ there is a path $P$ joining $u$ and $v$ in $G,$
\begin{equation*}
    u,\, v\in V(P) \,\, \text{and}\,\, P\subseteq G.
\end{equation*}

A connected graph without cycles is called a {\it tree}.

\begin{definition}
\label{wer}
A tree $T$ is a {\it star graph} if there is a vertex $c \in V(T)$, the {\it center} of $T$, such that $c$ and $v$ are adjacent for every $v \in V(T) \setminus \{c\}$ but for all $u, w \in T$ we have $\{u, w\} \notin E(T)$ whenever
    \begin{equation*}
    u\neq c\neq w.
    \end{equation*}
\end{definition}

\begin{definition}
\label{ty}
A {\it labeled tree} $T(l)$ is a pair $(T,l)$, where $T$ is a tree and $l$ is a mapping defined on the vertex set $V(T)$.
\end{definition}

  The two labeled trees $T_1(l_1)$ and $T_2(l_2)$ are said to be equal iff
  \begin{equation}
      T_1 = T_2\,\, {\rm and}\,\, l_1 = l_2.
  \end{equation}

In what follows, we consider only the nonnegative real-valued labelings $l\colon V(T) \to \mathbb{R}^+$.

Let $T=T(l)$ be a labeled tree. Following~\cite{Dov2020TaAoG}, we define a mapping $d_l \colon V(T) \times
V(T) \to \mathbb{R}^{+}$ as
\begin{equation}  \label{e1.1}
d_l(u, v) =
\begin{cases}
0, & \text{if } u = v, \\
\max\limits_{w \in V(P)} l(w), & \text{otherwise},%
\end{cases}%
\end{equation}
where $P$ is the path joining $u$ and $v$ in $T$.

The following result is a direct corollary of Proposition\,3.2 from \cite{Dov2020TaAoG}.

\begin{theorem}\label{t1.4}
Let $T = T(l)$ be a labeled tree. Then the
function $d_l$ is an ultrametric on $V(T)$ if and only if the inequality
\begin{equation}\label{t1.4_eq1}
\max\{l(u), l(v)\} > 0
\end{equation}
holds for every edge $\{u, v\}$ of $T$.
\end{theorem}

We will say that
$l $ is {\it non-degenerate} if \eqref{t1.4_eq1}
is valid for all $\{u,v\}\in E(T)$.

Let us introduce a class $\mathbf{US}$ of ultrametric spaces by the rule:
An ultrametric space $(X, d)$ belongs to $\mathbf{US}$ if and only if there is a labeled star graph $S(l)$ satisfying
\begin{equation}
    \label{mg}
    X=V(S) \,\, {\rm and} \,\, d=d_l,
\end{equation}
 where $d_l$ is defined as in \eqref{e1.1} with $T=S$.
If \eqref{mg} holds, then we say that $(X, d)$ is an {\bf US}-space generated by $S(l).$

We will also use the concept of isomorphic labeled trees.
Before introducing into consideration this concept, it is useful to remind the definition of isomorphism for free trees.

\begin{definition}
\label{mark}
Let $T_1$ and $T_2$ be trees. A bijection $f\colon V(T_1) \to V(T_2)$ is an {\it isomorphism} of $T_1$ and $T_2$ if
\begin{equation*}
(\{u, v\} \in E(T_1)) \Leftrightarrow (\{f(u), f(v)\} \in E(T_2))
\end{equation*}
is valid for all $u,$ $ v \in V(T_1)$. Two trees are {\it isomorphic} if there exists an isomorphism of these trees.
\end{definition}

For the case of labeled trees, Definition \ref{mark} must be modified as follows.

\begin{definition}
\label{mark1}
Let $T_1(l_1)$ and $T_2(l_2)$ be labeled trees. A mapping $f\colon V(T_1) \to V(T_2)$ is an isomorphism of $T_1(l_1)$ and $T_2(l_2)$ if it is an isomorphism of the free trees $T_1$ and $T_2$ and the equality
\begin{equation*}
l_2(f(v)) = l_1(v)
\end{equation*}
holds for every $v \in V(T_1)$.
\end{definition}

The main results are proven in Sections 2 and 3. In particular, Theorem~\ref{xz} gives us a metric characterization $\bf US$-spaces. The $\bf US$-spaces which admit a unique labeled star graph generating them are described up to isometry in Theorem~\ref{aaa}.

Theorem \ref{sp} shows that every self-isometry of an ultrametric space $(X,d)$ generated by labeled tree  $T_1(l_1)$
is an isomorphism $T_1(l_1)$ and a labeled tree $T_2(l_2)$
which also generate $(X, d)$.

In Theorem \ref{iuyr} we prove the necessary and sufficient conditions under which the group of self-isometries of $\bf US$-space coincides with
the group of self-isomorphism of every labeled star graph generating this space.

The final Section 4 contains a Conjecture \ref{sepjtg} that describes, up to weak similarity, the finite $\bf US$-spaces by some four-point condition. Using Example~\ref{£5^[77]} we show that this condition can be fulfilled for some infinite ultrametric space $(X,d) \notin {\bf US}$. In Conjecture \ref{tghjnk} we claim that each infinite, compact $\bf US$-space is the completion of an ultrametric space generated by labeled ray.

\section{Characterization of {\bf US}-spaces and unique-\\ness of generating star graphs}

\hspace{4mm}
The following theorem gives us a purely metric characterization of {\bf US}-spaces.

\begin{theorem}
    \label{xz}
Let $(X, d)$ be an ultrametric space. Then the following statements are equivalent:

{\rm (i)} $(X, d) \in {\bf US}.$

{\rm (ii)} There is $x_0 \in X$ such that the inequality
\begin{equation}
\label{zg}
    d(x_0, x) \leq d(y,x)
\end{equation}
holds whenever
\begin{equation}
\label{cd}
    x_0 \neq x \neq y.
\end{equation}

\end{theorem}

\begin{proof}
 (i) $\Rightarrow$ (ii). Let (i) hold. Then there is a labeled star graph $S(l)$ satisfying \eqref{mg}. Let us denote by $c$ a center of the star graph $S$. To prove statement (ii) it suffices to show that the inequality
\begin{equation}
\label{se}
    d_l(c, x) \leq d_l(y, x)
\end{equation}
 holds
 for all $x, y \in V(S)$ satisfying
 \begin{equation}
 \label{kj}
     c \neq x \neq y.
 \end{equation}

Let us consider arbitrary $x, y \in V(S)$ for which \eqref{kj} is valid. Then using Definition \ref{wer}, condition \eqref{kj} and formula \eqref{e1.1} we see that \eqref{se} holds iff
\begin{equation*}
    \max\{l(c), l(x)\} \leq \max\{l(y), l(c), l(x)\}.
\end{equation*}
The last inequality is trivially valid. Inequality \eqref{se} follows.

(ii) $\Rightarrow$ (i). Let statement (ii) hold. Let us denote by $S$ the star graph with $V(S) = X$ and the center
$c = x_0$, where $x_0$ is a point of $X$ satisfying  \eqref{zg} whenever \eqref{cd} holds. Now we define a labeling $l\colon V(S) \to \mathbb{R}^+$ as
\begin{equation}
    \label{oi}
l(x) =
\begin{cases}
0, & \text{if } x = x_0, \\
d(x_0, x), & \text{if } x \neq x_0.
\end{cases}
\end{equation}

Using Definition \ref{wer} and the inequality $d(x_0, x) > 0$, which is valid for each $x \neq x_0$ by definition of ultrametric spaces, we see that the labeling $l$ defined by \eqref{oi} is non-degenerate. Consequently the function $d_l$ defined by \eqref{e1.1} with $T = S$ is an ultrametric.  The ultrametric space $(X,d_l)$ belongs to {\bf US} by definition.
Thus to complete the proof it remains to show that
the equality
\begin{equation}
\label{2.6}
\quad d(x, y) = d_l(x, y)
\end{equation}
holds for all different $x, y \in X$.

First of all we note \eqref{oi} and \eqref{e1.1} imply \eqref{2.6} if $x = x_0$ or $y = y_0$.
Let
\begin{equation}
\label{pl}
    x \neq x_0 \neq y\,\, {\rm and} \,\,x \neq y
\end{equation}
hold.
Then the strong triangle inequality (see \eqref{d(x,y)}) and the equalities
\begin{equation}
\label{fra}
    d(x_0, x) = d_l(x_0, x), \quad  d(x_0, y) = d_l(x_0, y)
\end{equation}
imply
the inequality
\begin{equation*}
d(x, y) \leq \max\{d_l(x_0, x), d_l(x_0, y)\}.
\end{equation*}
It follows from  \eqref{e1.1}, \eqref{oi}, \eqref{pl} and \eqref{fra}   that
\begin{equation}
\label{28}
 \quad \max\{d_l(x_0, x), d_l(x_0, y)\} = d_l(x, y).
\end{equation}
Hence the inequality
\begin{equation}
\label{za}
d(x, y) \leq d_l(x, y)
\end{equation}
holds.
If \eqref{2.6} is not valid, then
it follows from \eqref{za} that
\begin{equation}
\label{fh}
    d(x,y) < d_l(x,y).
\end{equation}
Using \eqref{fra}, \eqref{28} and \eqref{fh}
we obtain
\begin{equation*}
d(x,y) < d_l(x_0,x) = d(x_0,x)
\end{equation*}
or
\begin{equation*}
d(x,y) < d_l(x_0,y) = d(x_0,y),
\end{equation*}
contrary to the definition of the point $x_0$.

The proof is completed.

\end{proof}

\begin{corollary}\label{zem}
 Let $(X, d)$ be an ultrametric space.
If the inequality
\begin{equation}
\label{9g}
    {\rm card}(X) \leq 3
\end{equation}
holds, then $(X, d)$ is an {\bf US}-space.
\end{corollary}

\begin{proof}
 To see this, notice that all ultrametric triangles are isosceles triangles with a base no larger than the "legs".
Thus, \eqref{9g} implies $(X, d) \in {\bf US}$ by Theorem \ref{xz} (see Figure 1).

\end{proof}

\begin{figure}[h!]\label{f1}
\centering

\begin{tikzpicture}
  \node[draw=none] (label1) at (1.3,2.2) {$(X,d)$};
    \node[draw, circle, fill=black, inner sep=1pt, label=above:{}] (A) at (0,2) {};
    \node[draw, circle, fill=black, inner sep=1pt, label=below:{}] (B) at (-1,0) {};
    \node[draw, circle, fill=black, inner sep=1pt, label=below:{}] (C) at (1,0) {};

    \draw (A) -- (B) node[midway, left] {$2$};
    \draw (A) -- (C) node[midway, right] {$2$};
    \draw (B) -- (C) node[midway, below] {$1$};
\end{tikzpicture}
\quad
\begin{tikzpicture}
  \node[draw=none] (label1) at (1.5,2) {$T(l)$};
    \node[draw, circle, fill=black, inner sep=1pt, label=above:{\textcircled{2}}] (O) at (0,2) {};
    \node[draw, circle, fill=black, inner sep=1pt, label=below:{\textcircled{0}}] (P1) at (-1,0) {};
    \node[draw, circle, fill=black, inner sep=1pt, label=below:{\textcircled{1}}] (P3) at (1,0) {};

    \draw (O) -- (P1);
    \draw (P1) -- (P3);
\end{tikzpicture}

\caption{The labeled star graph $T(l)$ generates the ultrametric triangle $(X, d)$.}
\end{figure}
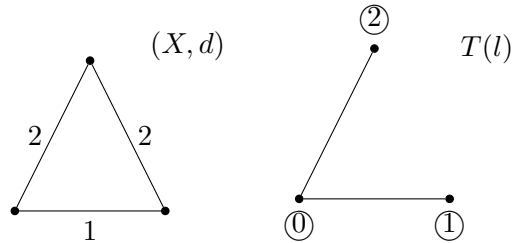

\begin{example}
    \label{fgjkdfb}
Let us define a mapping $d^+\colon \mathbb{R}^+ \times \mathbb{R}^+ \to \mathbb{R}^+$ as

\begin{equation}
\label{fdhinb}
d^+(p, q) =
\begin{cases}
0, & \text{if } p = q, \\
\max \{p, q\}, & \text{if } p \neq q.
\end{cases}
\end{equation}
Then $d^+$ is an ultrametric on $\mathbb{R}^+$ and, in addition, \eqref{fdhinb} gives the inequality
\begin{equation*}
d^+(0, p) \leq d(p, q)
\end{equation*}
whenever $0\neq p\neq q$.
Hence the ultrametric space $(\mathbb{R}^+, d^+)$ belongs  ${\bf US}$ by Theorem~\ref{xz}.

\end{example}

\begin{remark}
The ultrametric space $(\mathbb{R}^+, d^+)$ was introduced by Delhomme, Laflamme, Pouzet and Sauer in \cite{Delhommé}. The importance of this space and its subspaces was noted by Yoshito Ishiki in \cite{Ishiki}.
The endomorphisms of $(\mathbb{R}^+, d^+)$ are described in \cite{Dovgoshey-Kostikov-2024}.
\end{remark}

Let $(X, d)$ be an ultrametric space. Below we denote by $D(X)=D(X,d)$ the set of all distances between points of $(X, d)$,
\begin{equation}
\label{ew}
    D(X) := \{d(x, y) : x, y \in X\},
\end{equation}
    and write
\begin{equation}
\label{ew1}
D_0(X) = D_0(X,d) :=D(X,d)\setminus \{0\}.
\end{equation}

The next our result is a metric description of \textbf{US}-spaces which admit only one labeled star graph generated them.

\begin{theorem}
\label{aaa}
Let $(X, d)$ be an \textbf{US}-space. Then the following statements are equivalent:

   {\rm (i)} There exists a unique labeled star graph which generates $(X, d)$.

     {\rm (ii)}  If some labeled star graphs $ S_1(l_1) $ and $ S_2(l_2) $ generate $(X,d) $, then $S_1(l_1) $ are isomorphic to $S_2(l_2).$

    {\rm (iii)} The equality
    \begin{equation}
    \label{uy1}
    {\inf } D_0 (X) = 0
    \end{equation}
    holds.
\end{theorem}

\begin{proof}
Let us consider first the case
\begin{equation}
    \label{uy2}
\text{card}(X) > 1.
\end{equation}

(i) $\Rightarrow$ (ii). This implication is trivially true.

(ii) $\Rightarrow$ (iii). Let (ii) be valid. We must prove equality \eqref{uy1}.
Suppose contrary that
\begin{equation}
    \label{uy3}
{\inf } D_0(X) \neq 0.
\end{equation}
Then the inclusion $D_0(X) \subseteq \mathbb{R}^+$ and \eqref{uy3} give us the inequality
\begin{equation}
    \label{uy4}
\inf D_0(X) > 0.
\end{equation}

Using \eqref{uy4} we can find a positive number $t^*$ such that
\begin{equation}
    \label{uy5}
0 < t^* < \inf D_0(X).
\end{equation}

    Let $c$ be a center of some labeled star graph which generates $(X,d)$.

Similarly \eqref{oi} we define labelings
$l^*\colon V(S) \to \mathbb{R}^+$ and $l^0\colon V(S) \to \mathbb{R}$
by
\begin{equation}
\label{uy6}
l^*(x) =
\begin{cases}
t^*, & \text{if } x = c, \\
d(c, x), & \text{if } x \neq c,
\end{cases}
\end{equation}
and
\begin{equation}
\label{uy7}
l^0(x) =
\begin{cases}
0, & \text{if } x = c, \\
d(c, x), & \text{if } x \neq c
\end{cases}
\end{equation}
for all $x \in X$. Now we introduce the ultrametrics $d_{l^*}$ and $d_{l^0}$ on the set $X$ by rule \eqref{e1.1} with $T = S$ and $l = l^*$ and, respectively, $l = l^0$.

It follows from \eqref{uy5}, \eqref{uy6} and \eqref{uy7} that the equality
\begin{equation}
\label{uy8}
d_{l^0} = d_{l^*}
\end{equation}
holds. Reasoning in the same way as in proof of \eqref{2.6}, we obtain the equality
\begin{equation}
\label{uy9}
d_{l^0}=d.
\end{equation}

(We only note that \eqref{zg} holds with $x_0 = c$ whenever $c \neq x \neq y$).

Therefore,  $d_{l^*} = d$ also holds by \eqref{uy8}.
Thus the labeled star graphs $S(l^*)$ and $S(l^0)$ generate $(X,d)$ simultaneously.

Statement (ii) implies that
$S(l^*)$ and $S(l^0)$ are isomorphic labeled star graphs.
Hence, the equality
\begin{equation}
\label{uy10}
l^0 (V(S)) = l^* (V(S))
\end{equation}
holds by Definition \ref{mark1}.
Using \eqref{uy5}, \eqref{uy6} and \eqref{uy7} it is easy to see that
\begin{equation*}
0 \in l^0(V(S))\,\, \text{and}\,\,
0 \notin l^{*}(V(S)),
\end{equation*}
that contradicts \eqref{uy10}.

(iii) $\Rightarrow$ (i). Let equality \eqref{uy1} hold.

Let us consider two arbitrary labeled star graphs $S_1(l_1)$ and $S_2(l_2)$ which generate $(X,d)$.
To prove statement (i) it suffices to show that
\begin{equation}
\label{ku}
S_1 = S_2 \,\, \text{and} \,\, l_1 = l_2.
\end{equation}

Let us denote by $c_i$ a center of the star graph $S_i(l_i)$, $i = 1,2$.

It follows from Definition \ref{wer} and \eqref{mg} that the equality
\begin{equation}
\label{qw}
c_1 = c_2
\end{equation}
implies
\begin{equation}
\label{ro}
S_1 = S_2.
\end{equation}

To prove equality \eqref{qw}, note that equality \eqref{uy1} implies the existence of sequences $(x_n)_{n \in \mathbb{N}}$ and $(y_n)_{n \in \mathbb{N}}$ of points of $X$ such that
\begin{equation}
\label{pu}
x_n \neq y_n
\end{equation}
for every  $n \in \mathbb{N}$ and
\begin{equation}
\label{zi}
\lim_{n \to \infty} d(x_n, y_n) = 0.
\end{equation}

Now using inequality \eqref{se} with $x = x_n$, $y = y_n$ and $c = c_i, $ $l = l_i$,
we obtain
\begin{equation*}
d(c_i, x_n) \leq d(x_n, y_n)
\end{equation*}
for all $n\in \mathbb{N}$ and $i\in\{1,2\}.$
The last inequality and \eqref{zi} imply
\begin{equation}
\label{ty2}
 \lim_{n \to \infty} d(c_1,x_n)=\lim_{n \to \infty}d(c_2,x_n)=0.
\end{equation}
Hence the sequence $(x_n)_{n\in \mathbb{N}}$ is convergent in $(X,d).$
Since each convergent sequence has exactly one limit, equality \eqref{qw} follows from \eqref{ty2}.

Let us now prove the equalities
\begin{equation}
\label{lk}
l_1(c_1) = 0 \,\, \text{and} \,\, l_2(c_2) = 0.
\end{equation}

If the first equality in \eqref{lk} is false, then
we have
\begin{equation*}
l_1(c_1) = t
\end{equation*}
for some $t > 0$. The last equality and \eqref{e1.1} give us the inequality
\begin{equation}
\label{at}
d_{l_1}(c_1, x_n) \geq t
\end{equation}
for each $n \in \mathbb{N}$.
Since $(X, d)$ is generated by $S_1(l_1)$, we have the equality $d = d_{l_1}$, that together with \eqref{at} implies
\begin{equation*}
d(c_1, x_n) \geq t > 0
\end{equation*}
for each $n \in \mathbb{N}$, contrary to \eqref{ty2}.

The first equality in \eqref{lk} follows. The second one can be proved similarly.

Let us consider now an arbitrary point $x \in X$. Then the equalities $d_{l_1} = d,$ $d_{l_2} = d,$
\eqref{e1.1}, \eqref{qw} and \eqref{lk} imply
\begin{equation*}
l_1(x) = d(x, c_1) = d(x, c_2) = l_2(x).
\end{equation*}

The equality $l_1 = l_2$ is proved. Thus the implication
(iii) $\Rightarrow$ (i)  is true if \eqref{uy2} holds.

Let us consider now the case
\begin{equation*}
\text{card}(X) \leq 1.
\end{equation*}

Then the equality
\begin{equation}
\label{rp}
\text{card}(X) = 1
\end{equation}
holds, because for every ultrametric space $(X, d)$ the set $X$ is nonempty by definition.
Equalities \eqref{ew}, \eqref{ew1} and \eqref{rp} imply
\begin{equation*}
D(X) = \{0\}, \,\, \text{and}\,\, D_0(X) = \{0\}\setminus\{0\}= \emptyset.
\end{equation*}

Considering $D_0(X)$ as an empty subset of $[0, +\infty]$ and using the definition of infinum, we obtain the equality
\begin{equation*}
\inf D_0(X) = +\infty.
\end{equation*}

Thus statement (iii) is false.

If $S(l)$ is a labeled star graph generated $(X,d)$, then we have
\begin{equation*}
    \text{card}(V(S))=1
\end{equation*}
by \eqref{mg} and \eqref{rp}.
Let us denote by $c$ the unique vertex of $S.$
Then, for every $t\in [0,\infty)$, the labeling $l\colon V(S)\to \mathbb{R}^+$ defined by
\begin{equation*}
    l(c)=t,
\end{equation*}
satisfies the equality $d_l=d.$
Hence there exist infinitely many non-isomorphic labeled star graphs which generate $(X,d)$.

Thus statements (i) and (ii)  are also false and, consequently, the logical equivalences
\begin{equation*}
    {\rm (i)} \Leftrightarrow {\rm (ii)} \,\, \text{and}\,\, {\rm (i)} \Leftrightarrow {\rm (iii)}
\end{equation*}
are true.

The proof is completed.

\end{proof}

\begin{example}
Let $(\mathbb R^+,d^+)$ be the ${\bf US}$-space defined in Example~\ref{fgjkdfb}. Then the equality
\begin{equation*}
\inf D_0 (\mathbb R^+,d^+)=0
\end{equation*}
holds. Hence, by Theorem \ref{aaa}, there is a unique labeled star graph generating $(\mathbb R^+,d^+).$
\end{example}

\section{From isometries to isomorphism  and back}

\hspace{4mm} The following proposition coincides with Lemma 3.11 from \cite{Dov2020TaAoG}.

\begin{proposition}
    \label{ljh}
If $f \colon V(T_1) \to V(T_2)$ is an isomorphism of labeled trees $T_1 = T_1(l_1)$ and $T_2 = T_2(l_2)$, then the equality
\begin{equation*}
d_{l_1}(u_1, v_1) = d_{l_2}(f(u_1), f(v_1))
\end{equation*}
holds for all $u_1, v_1 \in V(T_1)$.
\end{proposition}

The converse of Proposition \ref{ljh} is generally not true. Even if $T_1(l_1)$ and $T_2(l_2)$ are not isomorphic as free graphs, the ultrametric spaces $(V(T_1), d_{l_1})$ and $(V(T_2), d_{l_2})$ can be isometric. An example of such trees is shown in Figure 2 
below.

\begin{figure}[h!]\label{f2}
\centering
\begin{tikzpicture}[every node/.style={circle, draw, minimum size=0.6cm, font=\small}, node distance=1cm, scale=0.8, transform shape]
    \node[draw=none] (label1) at (1.5,1.2) {$T_1(l_1)$};
    \node (a1) at (0,0) {$a_1$};
    \node (a2) [right of=a1, node distance=1.3cm] {$a_2$};
    \node (a3) [right of=a2, node distance=1.3cm] {$a_3$};
    \node (a4) [right of=a3, node distance=1.3cm] {$a_4$};

    \draw (a1) -- (a2);
    \draw (a2) -- (a3);
    \draw (a3) -- (a4);

    \node[draw=none] (label2) at (9,1.2) {$T_2(l_2)$}; 
    \node (b1) at (7.5,0) {$a_1$};
    \node (b2) [above=0.4cm of b1] {$a_3$};
    \node (b3) [left=0.4cm of b1] {$a_4$};
    \node (b4) [right=0.4cm of b1] {$a_2$};

    \draw (b1) -- (b2);
    \draw (b1) -- (b3);
    \draw (b1) -- (b4);
\end{tikzpicture}
\caption{The labeled path $T_1(l_1)$ and the labeled star graph $T_2(l_2)$ generate isometric spaces $(V(T_1), d_{l_1})$ and $(V(T_2), d_{l_2})$ if $0 <a_1 \leq a_2 \leq a_3 \leq a_4$.}
\end{figure}
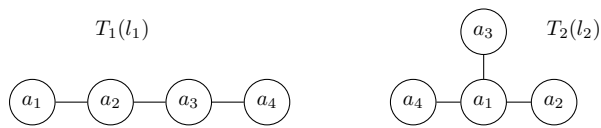

The next Proposition \ref{tre} shows that the isometry
of \textbf{US}-spaces implies that  the star graphs
generated  these spaces are isomorphic as free trees.

\begin{proposition}
\label{tre}
Let $(X_1, d_1)$ and $(X_2, d_2)$ be isometric \textbf{US}-spaces and let labeled
star-graphs $S_{1}(l_1)$ and $S_{2}(l_2)$ generate
$(X_1, d_1)$ and $(X_2, d_2)$ respectively.
Then the free star graphs $S_1$ and $S_2$ are isomorphic.
\end{proposition}

\begin{proof}
Since every isometry is a bijection,  we have the equality
\begin{equation*}
\text{card}(X_1) = \text{card}(X_2).
\end{equation*}
The last equality and \eqref{mg} imply
\begin{equation}
    \label{dfg}
\text{card}(V(S_1)) = \text{card}(V(S_2)).
\end{equation}
It follows from \eqref{dfg}, Definition \ref{wer} and Definition \ref{mark}  that $S_1$ and $S_2$ are
isomorphic.
\end{proof}

The results in the rest of the section are mainly motivated by the following problem cf. Problems 3.1 and 3.2 from \cite{Dov2020TaAoG}.
\begin{problem}
    Let $T_1(l_1)$ and $T_2(l_2)$ be labeled trees generating ultrametric spaces $(X_1, d_1)$ and $(X_2, d_2)$ respectively. Find conditions under which the statements:
\begin{itemize}
    \item $T_1(l_1)$ and $T_2(l_2)$ are isomorphic
         \item $(X_1, d_1)$ and $(X_2, d_2)$ are isometric
     \end{itemize}
logically equivalent.
\end{problem}

\begin{theorem}
\label{sp}
Let $(X, d)$ be an ultrametric space generated by  labeled tree $T_1(l_1)$ and let $F \colon X \to X$ be a bijective mapping. Then the following statements are equivalent:

     {\rm (i)} $F$ is a self-isometry of $(X, d)$.

     {\rm (ii)} There is a labeled tree $T_2(l_2)$ such that $(X, d)$ is generated by $T_2(l_2)$ and $F$ is an isomorphism of $T_1(l_1)$ and $T_2(l_2)$.

\end{theorem}

\begin{proof}
 (i) $\Rightarrow$ (ii). Let (i) hold. Since  $(X, d)$ is generated by $T_1(l_1)$, we have the equality
\begin{equation}
\label{bn}
d= d_{l_1},
\end{equation}
where $d_{l_1}$ is defined as in \eqref{e1.1} with $T = T_1$.  Since $F$ is bijective, there is the inverse mapping $F^{-1}\colon X \to X$ of $F$.

Let us define a labeled tree $T_2 (l_2)$ such that $V(T_2) = X$, and
\begin{equation*}
\left(\{u, v\} \in E(T_2)\right)  \Leftrightarrow \left(\left\{F^{-1}(u), F^{-1}(v)\right\} \in E(T_1)\right),
\end{equation*}
and
\begin{equation*}
l_2(w) = l_1\left(F^{-1}(w)\right)
\end{equation*}
for all $u, v, w \in V(T_2)$.
Definitions \ref{mark} and \ref{mark1} imply that $F$ is an isomorphism of labeled trees $T_1(l_1)$ and $T_2(l_2)$.

Hence (ii) holds if $T_2(l_2)$ generates $(X, d)$, i.e., if we have the equality
\begin{equation}
\label{er}
d = d_{l_2}.
\end{equation}

Equality \eqref{bn} shows that \eqref{er} holds iff we have
\begin{equation}
\label{xi}
d_{l_1}(x, y) = d_{l_2}(x, y)
\end{equation}
for all $x, y \in X$.

Proposition \ref{ljh} with $V(T_{1}) = V(T_{2}) = X$, $f = F$ gives us the equality
\begin{equation}
\label{rtu}
d_{l_1}(u, v) = d_{l_2}(F(u), F(v))
\end{equation}
for all $u, v \in V(T_{1})$.
Since $F$ is a self-isometry of $(X, d)$ and \eqref{bn} holds, $F$ also is a self-isometry of $(V(T_{1}), d_{l_{1}})$.
Hence we may rewrite \eqref{rtu} as
\begin{equation}
\label{rtu1}
d_{l_1}(F(u), F(v)) = d_{l_2}(F(u), F(v)).
\end{equation}

Let us consider now arbitrary $x, y \in X$. Then the bijectivity of $F$ implies that there are
$u, v \in V(T)$ satisfying
\begin{equation}
\label{ssl}
x = F(u) \,\, \text{and} \,\, y = F(v).
\end{equation}

Now \eqref{xi} follows from \eqref{rtu1} and \eqref{ssl}.

(ii) $\Rightarrow$ (i). Let $T_{2}(l_2)$ be a labeled star graph generating $(X, d)$ and let $F\colon V(T_{1}) \to V(T_{2})$
be an isomorphism of $T_{1}(l_1)$ and $T_{2}(l_2)$.

Then the equalities
\begin{equation}
\label{klt}
d = d_{l_{1}}, \,\, d = d_{l_{2}},
\end{equation}
and
\begin{equation}
\label{gna}
    V(T_{1}) = X = V(T_{2})
\end{equation}
hold. Using Proposition \ref{ljh} and formulas \eqref{klt}, \eqref{gna} we obtain
\begin{equation*}
d(u, v) = d_{l_{1}}(u, v) = d_{l_{2}}(F(u), F(v))= d(F(u), F(v))
\end{equation*}
for all $u, v \in X$. Thus $F$ is a self-isometry of $(X, d)$ by Definition \ref{am}.

The proof is completed.

\end{proof}

\begin{corollary}
\label{twin} Let $(X, d)$ be an {\bf US}-space and let a labeled star graph $S_1(l_1)$ generate $(X, d)$. Then for every self-isometry $F\colon X \to X$ of $(X, d)$ there is a labeled star graph $S_2(l_2)$ such that $S_2(l_2)$ generates $(X, d)$ and $F$ is an isomorphism of $S_1(l_1)$ and $S_2(l_2)$.
\end{corollary}

In what follows we denote   by $ \text{Iso }(X,d) $ the set of all
self-isometries of the
ultrametric space $ (X,d) $
and by $ \text{Iso }S(l) $ the set of
all self-iso\-morp\-hisms
of labeled star graph $ S(l) $.

Corollary \ref{twin} and Theorem \ref{aaa} give as the following.

\begin{corollary}\label{kram}
Let $(X, d)$ be an \textbf{US}-space and let $S(l)$ be
a labeled star graph generating $ (X,d) $.
Then the equality
\begin{equation*}
\inf D_0(X) = 0
\end{equation*}
implies the equality
\begin{equation}
\label{fgk}
\text{Iso }(X,d) = \text{Iso }S(l).
\end{equation}
\end{corollary}

\begin{example}\label{ek}
 Let $(X, d)$ be a two-point \textbf{US}-space, $X = \{x_1, x_2\}$.
Then the labeled star graph $S(l)$ with $V(S) = X$ and
\begin{equation*}
 l(x) =
 \begin{cases}
d(x_1,x_2), & \text{if } x = x_1, \\
0, & \text{if } x = x_2,
\end{cases}
\end{equation*}
generates $(X, d)$, but we have
\begin{equation*}
F \in \text{Iso }(X, d) \,\, \text{and} \,\, F \notin \text{Iso }(S(l))
\end{equation*}
for  $F\colon X \to X$ defined by
\begin{equation*}
F(x_1) = x_2 \,\, \text{and} \,\, F(x_2) = x_1.
\end{equation*}
\end{example}

The next theorem shows that equality \eqref{fgk} can be satisfied even if $\inf D_0(X)>0.$

\begin{theorem}
    \label{iuyr}
 Let $(X,d)$ be an ${\bf US}$-space. Then the following statements are equivalent:

    {\rm (i)} The equality
    \begin{equation*}
        \text{Iso }(X,d) = \text{Iso }S(l)
    \end{equation*}
    holds for each labeled star graph $S(l)$ generating $(X,d)$.

 {\rm (ii)} The set $D_0$ has no least element.
\end{theorem}

\begin{proof}

Let us consider first the case when
\begin{equation}
\label{dfgg}
    \text{card}(X) = 1.
\end{equation}

Then $X$ is a singleton set, $X = \{ x_0 \}$, so that all elements of $\text{Iso }(X,d)$ and $\text{Iso } S(l)$ coincide with the mapping
\begin{equation*}
    X \ni x_0 \mapsto x_0 \in X.
\end{equation*}
Hence (i) is true.
Moreover, if \eqref{dfgg} holds, then the set $D_0$ is empty by \eqref{ew1}. Thus (ii) is also true.

Let us consider now the case when
\begin{equation}
\label{fgjutt}
 \text{card}(X) \geq 2.
\end{equation}

(i) $\Rightarrow$ (ii).
Let (i) hold. We must show that statement (ii) also holds.

As in the proof of Theorem \ref{xz} (see formula \eqref{oi}), we can find a labeled star graph $S(l)$ generating $(X,d)$ such that
\begin{equation}
\label{pii}
l(c) = 0,
\end{equation}
where $c$ is a fixed center of $S$.
Then equality \eqref{e1.1} with $T(l) = S(l)$, and formulas \eqref{ew}--\eqref{ew1}, and formula \eqref{pii} give us the equality
\begin{equation}
\label{zaq}
D_0(X) = \{ l(u) \colon u \in V(S) \setminus \{c\} \}.
\end{equation}

If (ii) false, then \eqref{zaq} implies that there is $v^* \in V(S) \setminus \{c\}$ such that
\begin{equation}
\label{sdwe}
l(v^*) \leq l(u)
\end{equation}
for each $u \in V(S) \setminus \{c\}$.

Let us define a mapping $F\colon X \to X$ by the rule
\begin{equation}
\label{mnngf}
F(x) =
\begin{cases}
v^*, & \text{if } x = c, \\
c, & \text{if } x = v^*, \\
x, & \text{if } c \neq x \neq v^*.
\end{cases}
\end{equation}
Since $\{v^*, c\}$ is an edge of $S$ and $l\colon V(S) \to \mathbb R^+$ is non-degenerate,
inequality \eqref{t1.4_eq1} and equality \eqref{sdwe} imply $l(v^*) > 0.$ Hence we have
\begin{equation}
\label{jgfdv}
l(F(c)) = l(v^*) > 0 = l(c)
\end{equation}
by \eqref{fgjutt} and \eqref{pii}. Thus $F$ is
not a self-isomorphism of $S(l)$,
\begin{equation}
\label{piet}
F \notin \text{Iso } S(l).
\end{equation}

We claim that the membership relation
\begin{equation}
\label{vhnsg}
F \in \text{Iso } (X,d)
\end{equation}
is true.

Let $u$ and $v$ be arbitrary non-equal points of $X$.
To prove \eqref{vhnsg} it suffices to show that the equality
\begin{equation}
\label{zqki}
d(u,v) = d(F(u), F(v))
\end{equation}
holds.
If we have
\begin{equation*}
\{u,v\} = \{v^*, c\} \,\, \text{or} \,\, \{u,v\} \cap \{v^*, c\} = \emptyset,
\end{equation*}
then \eqref{zqki} follows from \eqref{mnngf}.

Let us consider the case when the set
$
\{u,v\} \cap \{v^*, c\}
$
contains only one point.
Then, without loss of generality, we assume
\begin{equation*}
u \notin \{v^*, c\}
\end{equation*}
and
\begin{equation*}
v = c \,\, \text{or} \,\, v = v^*.
\end{equation*}

If we have $v=c,$ then \eqref{zqki} holds iff
\begin{equation}
\label{jhbeg}
d(u, c) = d(F(u), F(c)).
\end{equation}
Analogously, if we have $v=v^*,$ then \eqref{zqki} holds iff
\begin{equation}
\label{fhjoe}
d(u, v^*) = d(F(u), F(v^*)).
\end{equation}

Using \eqref{mnngf}, we rewrite \eqref{jhbeg}--\eqref{fhjoe} as the following single equality
\begin{equation}
\label{rggyji}
d(u, c) = d(u, v^*).
\end{equation}
Since $ T(l) $ generates $ (X, d) $, we also have the equality $ d = d_l $.
Thus \eqref{rggyji} holds iff
\begin{equation}
\label{ujrfv}
d_l(u, c) = d_l(u, v^*).
\end{equation}
Now using \eqref{e1.1} and \eqref{pii} we obtain
\begin{equation}
\label{sakl}
d_l(u, c) = \max \{ l(u), l(c) \} = l(u).
\end{equation}
Analogously \eqref{e1.1} and \eqref{sdwe} give us
\begin{equation}
\label{rybn}
d_l(u, v^*) = \max \{ l(u), l(v^*) \} = l(u).
\end{equation}

Equality \eqref{ujrfv} follows from \eqref{sakl} and \eqref{rybn}.
Thus, by formulas \eqref{piet} and \eqref{vhnsg} we have
\begin{equation*}
F \in \text{Iso }(X, d) \setminus \text{Iso } S(l),
\end{equation*}
which contradicts statement (i).

(ii) $\Rightarrow$ (i). Let statement (ii) hold.
We must show that the equality
\begin{equation}
\label{dgnuj}
\text{Iso } (X, d) = \text{Iso } S(l)
\end{equation}
holds for all labeled star graphs $S(l)$ generating $(X,d).$

If the equality  $ \inf D_0(X) = 0$
holds, then \eqref{dgnuj} follows from Corollary\,\ref{kram}.

Let us consider now the case
\begin{equation*}
\inf D_0(X) > 0.
\end{equation*}

Let $S(l)$ be a labeled star graph generating $(X,d)$. Statement (ii) implies that $X$ contains infinitely many points and, consequently, the free star graph $S$ has the unique center $c$.

It follows from Definition\,\ref{wer} and formula \eqref{e1.1} with $T(l) = S(l)$ that
\begin{equation}
\label{fbjls}
l(c)\leq \inf D_0(X,d_{l}).
\end{equation}
Since $S(l)$ generates $(X,d)$, the equality $d = d_{l}$ holds. Hence we may rewrite \eqref{fbjls} as
\begin{equation}
\label{iehonl}
l(c) \leq \inf D_0(X).
\end{equation}

Let us define now a new labeling $l^*\colon V(S) \to \mathbb{R}^+$ as
\begin{equation}
\label{iiyopp}
    l^*(u) =
    \begin{cases}
        0, & \text{if } u = c, \\
        l(u)-\inf D_0(X,d), & \text{if } u \neq c.
    \end{cases}
\end{equation}

Then Definition \ref{mark} and Definition\ref{mark1} give us the equality
\begin{equation}
\label{eftgyh}
    \text{Iso }S(l)= \text{Iso } S(l^*).
\end{equation}

Let us also define a new ultrametric $d^*$ on $X$ by the rule
\begin{equation}
\label{ddrgkk}
    d^*(x,y) =
    \begin{cases}
        d(x,y) - \inf D_0(X,d), & \text{if } x \neq y, \\
        0, & \text{if } x = y.
    \end{cases}
\end{equation}

To prove that $d^*$ is really an ultrametric on $X$ we note that statement (ii) gives us the inequality
\begin{equation*}
    d(x,y) > \inf D_0(X,d)
\end{equation*}
for all distinct points $x, y \in X$ and that the strong triangle inequality for $d^*$ follows from the strong triangle inequality for $d$.

Since for every nonempty $A \subseteq \mathbb{R}^+$ and each $t \in [0, \inf A]$ we have the equality
\begin{equation}
t+\inf\{a-t: a\in A\}=\inf A,
\end{equation}
formulas \eqref{ew}, \eqref{ew1} and \eqref{eftgyh} imply
\begin{equation}
\label{weevbg}
    D_0(X,d^*) = 0.
\end{equation}
Moreover, using \eqref{ddrgkk} it is easy to prove  the equality
\begin{equation}
\label{mooi}
    \text{Iso }(X, d^*) = \text{Iso }(X, d).
\end{equation}
Equalities \eqref{eftgyh} and \eqref{mooi} show that \eqref{dgnuj} holds
iff
\begin{equation}
\label{zcwwko}
\text{Iso }(X,d^*) = \text{Iso }S(l^*).
\end{equation}
Using Corollary \ref{kram} and equality \eqref{weevbg}, we see that to prove equality \eqref{zcwwko} it is enough to check that $S(l^*)$ generates $(X,d^*)$,  i.e.,
\begin{equation*}
d_{l^*} = d^*.
\end{equation*}
 The last equality follows the equality $d = d_l$,
\eqref{iiyopp}, \eqref{ddrgkk} and formula \eqref{e1.1} with $T(l) = S(l^*)$.

The proof is completed.

\end{proof}

\begin{example}
Let $X = \{0\} \cup (1, \infty)$ and let $S(l_t)$ be a labeled star graph with the center $c = 0$, the vertex set
\begin{equation*}
    V(S) = X,
\end{equation*}
and the labeling  $l_t\colon V(S) \to \mathbb{R}^+$
defined by
\begin{equation*}
l_t(x) =
\begin{cases}
t, & \text{if } x = 0, \\
x, & \text{otherwise},
\end{cases}
\end{equation*}
where $t$ is a point of the closed interval $[0,1]$.
Let us denote by $d$ the restriction $ d^+_{|X\times X}$ of the ultrametric of $d^+$ defined in Example~\ref{fgjkdfb}. Then, for every $t \in [0,1]$, the equality
\begin{equation*}
\text{Iso }(X, d) = \text{Iso } S(l_t)
\end{equation*}
holds by Theorem \ref{iuyr}.
\end{example}

\begin{remark}
Theorems \ref{aaa} and \ref{iuyr} show that the properties of the range set of ultrametrics  are important for descriptions of labeled star graphs representing ${\bf US}$-spaces.
It is interesting to note that compactness and separability of ultrametrizable topological spaces $(X,\tau)$ can be also described via properties of range set of compatible with topology $\tau$ ultrametrics \cite{Dovg-Sh}.
\end{remark}

\section{Two conjectures}

\hspace{ 4mm }
It was noted in Corollary \ref{zem} that every ultrametric space $(X,d)$ with $\text{card}(X) \leq 3$ is an {\bf US}-space. Below we will use the four-point ultrametric spaces which do not belong to {\bf US}.

\vspace{-5mm}

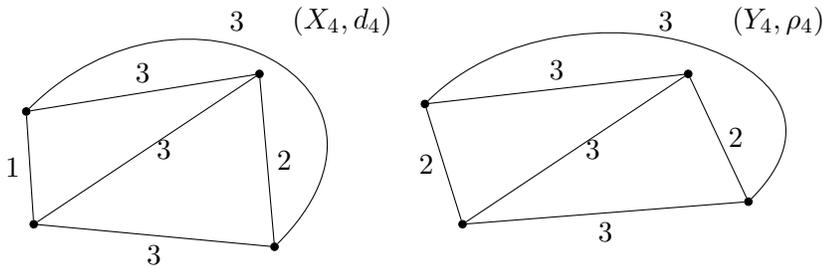
\begin{figure}[H]\label{cis}
\centering
\begin{tikzpicture}[remember picture]

  \node[draw=none] (label1) at (2.7,2.7) {$3$};
  \node[draw=none] (label1) at (4.1,2.7) {$(X_4,d_4)$};
    \node[draw, circle, fill=black, inner sep=1pt] (A) at (0,0) {};
    \node[draw, circle, fill=black, inner sep=1pt] (B) at (3.2,-0.3) {};
    \node[draw, circle, fill=black, inner sep=1pt] (C) at (-0.1,1.5) {};
    \node[draw, circle, fill=black, inner sep=1pt] (D) at (3,2) {};

    \draw (A) -- (B) node[midway, below] {$3$};
    \draw (A) -- (C) node[midway, left] {$1$};
    \draw (A) -- (D) node[midway, right] {$3$};
    \draw (C) -- (D) node[midway, above] {$3$};
    \draw (B) -- (D) node[midway, right] {$2$};
    \draw (B) to[out=45,in=45,looseness=2] (C) node[midway,right] {};

  \begin{scope}[xshift=5.7cm]
    \node[draw=none] (label2) at (2.7,2.7) {$3$};
    \node[draw=none] (label2) at (4.2,2.7) {$(Y_4,\rho_4)$};
    \node[draw, circle, fill=black, inner sep=1pt] (A) at (0,0) {};
    \node[draw, circle, fill=black, inner sep=1pt] (B) at (3.8,0.3) {};
    \node[draw, circle, fill=black, inner sep=1pt] (C) at (-0.5,1.6) {};
    \node[draw, circle, fill=black, inner sep=1pt] (D) at (3,2) {};

    \draw (A) -- (B) node[midway, below] {$3$};
    \draw (A) -- (C) node[midway, left] {$2$};
    \draw (A) -- (D) node[midway, right] {$3$};
    \draw (C) -- (D) node[midway, above] {$3$};
    \draw (B) -- (D) node[midway, right] {$2$};
    \draw (B) to[out=45,in=45,looseness=1.5] (C) node[midway,right] {};
  \end{scope}

\end{tikzpicture}
\caption{ $(X_4,d_4)$ and $(Y_4,\rho_4)$ are not ${\bf US}$-spaces.}
\end{figure}

\vspace{-4mm}

\begin{example}\label{exp}
The four-point ultrametric spaces $(X_4, d_4)$ and $(Y_4,\rho_4)$ depicted in Figure~3 
cannot be generated by any labeled star graphs. This follows easily from Theorem \ref{xz}.
\end{example}

Let us recall now the concept of weakly similar ultrametric spaces.
\begin{definition}
\label{scguite}
Let $(X, d)$ and $(Y, \rho)$ be ultrametric spaces.
 A bijective mapping $\Phi\colon X \to Y$ is a {\it weak similarity} if there is a strictly increasing function $f\colon D(Y) \to D(X)$ such that the equality
\begin{equation*}
d(x, y) = f \left( \rho \left( \Phi(x), \Phi(y) \right) \right)
\end{equation*}
holds for all $x, y \in X$.
\end{definition}

 The authors believe that the following conjecture is true.
\begin{conjecture}
\label{sepjtg}
The following statements are equivalent for every finite ultrametric space $(X^*,d^*)$:

{\rm (i)} $(X^*, d^*) \notin {\bf US}.$

{\rm (ii)} $(X^*, d^*)$  contains a four-point subspace which is weakly similar either to  $(X_4,d_4)$ or  to $(Y_4,\rho_4)$.
\end{conjecture}

\begin{remark}
The concept of weak similarity was introduced in \cite{Dovgoshey-Petrov} for the case of semimetric spaces.
\end{remark}

The next examples shows that statements (i) and (ii) of Conjecture \ref{sepjtg} are, in general, not equivalent if $X^*$ is an infinite set.

\begin{example}
    \label{£5^[77]}
The ultrametric space $(V(R),d_l)$ generated by labeled ray $R(l)$ admits an isometric embedding in the space $(V(S),d_{l^*})$ but $(V(R),d_l)\notin {\bf US}$ by Theorem \ref{xz} (see Figure~4).
\end{example}

\begin{figure}[h!]
    \centering
    \begin{tikzpicture}[every node/.style={scale=0.9, inner sep=0pt, outer sep=0pt}]
            \draw (-1,1) -- (0,1) -- (1,1);
            \node at (1.72,1) {$\dots$};
            \draw (2.35,1) -- (3.25,1) -- (3.7,1) -- (4.15,1);
             \node at (4.48,1) {$\dots$};

            \foreach \n in {-1,0,1}
                \node[circle,draw,fill=black,inner sep=1pt] at (\n,1) {};

            \node[circle,draw,fill=white,inner sep=3.6pt] at (-1,0.46) {$1$};
            \node[circle,draw,fill=white,inner sep=1.8pt] at (0,0.46) {$\frac{1}{2}$};
            \node[circle,draw,fill=white,inner sep=1.8pt] at (1,0.46) {$\frac{1}{3}$};

            \node at (1.72,0.46) {$\dots$};

            \foreach \n in {2.35,3.25,4.15}
                \node[circle,draw,fill=black,inner sep=1pt] at (\n,1) {};

            \node [circle,draw,fill=white,inner sep=1.8pt] at (2.35,0.46) {$\frac{1}{n}$};
            \node [circle,draw,fill=white,inner sep=-0.27pt] at (3.25,0.46) {$\frac{1}{n+1}$};
            \node [circle,draw,fill=white,inner sep=-0.27pt] at (4.15,0.46) {$\frac{1}{n+2}$};

            \node at (1.6,1.4) {$R(l)$};

        \node[circle,draw] (o) at (7,0) {0};

        \node[circle,draw,fill=white,inner sep=2pt] at (9,0) (n1) {$1$}; 
        \node[circle,draw,fill=white,inner sep=1pt] at (8.7,0.9) (n2) {$\frac{1}{2}$};
        \node[circle,draw,fill=white,inner sep=1pt] at (8.2,1.6) (n3) {$\frac{1}{3}$};
        \node[circle,draw,fill=white,inner sep=1pt] at (7.5,2.1) (n4) {};
        \node[circle,draw,fill=white,inner sep=1pt] at (7,2.3) (n5) {};
        \node[circle,draw,fill=white,inner sep=1pt] at (6.5,2.1) (n6) {$\frac{1}{n}$};
        \node[circle,draw,fill=white,inner sep=0.2pt] at (5.8,1.6) (n7) {$\frac{1}{n+1}$};
        \node[circle,draw,fill=white,inner sep=0.2pt] at (5.3,0.9) (n8) {$\frac{1}{n+2}$};
        \node[circle,draw,fill=white,inner sep=1pt] at (5,0) (n9) {};
        \node[circle,draw,fill=white,inner sep=1pt] at (5.3,-0.9) (n10) {};

        \draw (o) -- (n1);
        \draw (o) -- (n2);
        \draw (o) -- (n3);
        \draw[dashed] (o) -- (n4);
        \draw[dashed] (o) -- (n5);
        \draw (o) -- (n6);
        \draw (o) -- (n7);
        \draw (o) -- (n8);
        \draw[dashed] (o) -- (n9);
        \draw[dashed] (o) -- (n10);

        \node at (8.5,2.6) {$S(l^*)$};

    \end{tikzpicture}
    \caption{$(V(S),d_{l^*}) \in {\bf US}$ is a completion of $(V(R),d_l)\notin {\bf US}$.}
\end{figure}
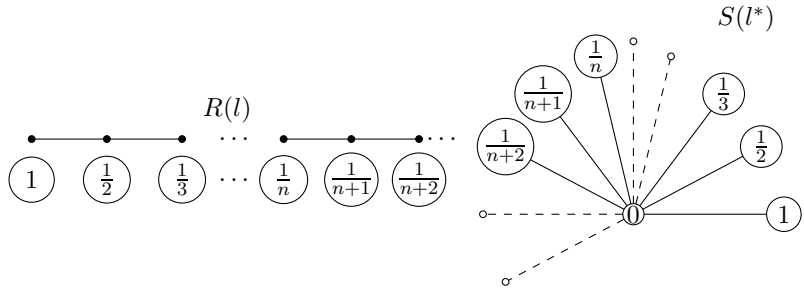

\begin{conjecture}
\label{tghjnk}
    Let $(X,d)$ be an infinite ${\bf US}$-space generated by labeled star graph with a center $c$,
    let $X_0 := X \setminus \{c\}$ and let $d_0$ be the restriction of $d$ on the set $X_0\times X_0$.
    Then the following statements are equivalent:

    {\rm(i)} $(X, d)$ is compact.

   {\rm (ii)} There is a labeled ray $R(l)$ generating $(X_0, d_0)$ such that
    \begin{equation*}
        E(R) = \{ \{ x_1, x_2\}, \{ x_2, x_3\}, \dots, \{x_n, x_{n+1}\}, \dots \}
    \end{equation*}
    and
    \begin{equation*}
        l(x_n) \geq l(x_{n+1}) > 0
    \end{equation*}
    for every integer $n \in \mathbb{N}$, and the limit relation
    \begin{equation*}
        \lim\limits_{n \to \infty} l(x_n) = 0
    \end{equation*}
    holds.

\end{conjecture}

\section*{Declarations}

\textbf{Declaration of competing interest}

 The authors declare no conflict of interest.

\medskip

\textbf{Data availability}

  All necessary data are included into the paper.

\medskip

\textbf{Acknowledgement}

The first author was supported by grant $359772$ of the Academy of Finland.

\medskip


\textbf{CRediT (Contributor Roles Taxonomy) authorship contribution statement}

Author1: Conceptualization, Methodology, Original Draft Preparation.

Author2: Visualization, Review-Editing, Software, Validation.

\bigskip

CONTACT INFORMATION

\medskip
Oleksiy Dovgoshey\\
Institute of Applied Mathematics and Mechanics of NASU, Slovyansk, Ukraine,\\
Department of Mathematics and Statistics, University of Turku, Turku, Finland \\
oleksiy.dovgoshey@gmail.com, oleksiy.dovgoshey@utu.fi

\medskip
Olga Rovenska\\
Department of Mathematics and Modelling, Donbas State Engineering Academy, Kramatorsk, Ukraine\\
rovenskaya.olga.math@gmail.com
\end{document}